\documentclass[10pt]{article}

\usepackage{amsfonts}
\usepackage{amssymb}
\usepackage{amsmath}
\usepackage{amsthm}
\usepackage{mathrsfs}
\usepackage{bbm}
\usepackage{appendix}

\theoremstyle{plain}
\newtheorem{theorem}{Theorem}[section]
\newtheorem{lemma}[theorem]{Lemma}
\newtheorem{corollary}[theorem]{Corollary}

\theoremstyle{definition}
\newtheorem{definition}[theorem]{Definition}

\theoremstyle{remark}
\newtheorem{remark}{Remark}

\title{Characterization and Analysis of Generalized Grey Incomplete Gamma Noise}

\author{W.~Bock$^1$, L. Cristofaro$^2$\\[.3cm]
$^1$ {\small Rheinland-Pf\"alzische Technische Universit\"at Kaiserslautern-Landau,}\\{\small  Fachbereich Mathematik,}\\ {\small Gottlieb-Daimler-Stra{\ss}e 48,}\\ {\small 67663 Kaiserslautern, Germany}\\
{\small E-Mail: bock@mathematik.uni-kl.de}\\[.2cm]
$^2$ {\small Sapienza University of Rome, Italy }\\
{\small E-Mail: lorenzo.cristofaro@uniroma1.it}}

\begin{document}

\maketitle

\vspace{10pt}

\begin{abstract}
	The grey incomplete gamma distributions was established by one of the authors in a previous publication. In this article we use the Kondratiev characterization theorem to identify those via a suitable Laplace transform with holomorphic functions with suitable properties. We establish theorems for the integration and convergence of sequences of these distributions. As direct applications of these analytic tools we give the examples of Donsker's delta function, the local time, identify the time-derivative of the process as a suitable distribution and define the Gamma grey Ornstein-Uhlenbeck process.
\end{abstract}
%
%
%
%
%

\section{Introduction}
\label{sec:introduction}

During the last decades white noise analysis has evolved into an infinite dimensional distribution theory, with rapid developments in mathematical structure and applications in various domains, see e.g.~the monographs \cite{Hid93, O94, Kuo96, HS17}. Various characterization theorems \cite{PS91, KLPSW96, GNM22} are proven to build up a strong analytical foundation. Fractional Brownian motion, with its specific properties, such as short/long range dependence and self-similarity, with natural applications in different fields (e.g. mathematical finance, telecommunications engineering, etc.), has a very natural representation in white noise theory, see \cite{Oks03, Mis08}.

From the end of the Eighties and in the Ninties, these methods were generalized to a non-Gaussian infinite dimensional analysis, by transferring properties of the Gaussian measure to e.g.~ the Poisson measure \cite{I88,KSU, Oks04}, and other measures \cite{Sch90, ADKS96, Kon98}.

Almost at the same time, the subject of fractional calculus and its applications (like Riemann-Liouville integral and Caputo fractional derivatives) has gained considerable popularity and importance mainly due to its applications in diverse fields of science and engineering (e.g. \cite{skm93}, \cite{Scal12}) and also fractional stochastic differential equations driven by white noise, see \cite{HO22}. These operators and their generalizations have been used to model problems with anomalous dynamics in different settings (e.g. \cite{Toa15}, \cite{Cap21}, \cite{SEO07} and \cite{Asc21}).

By investigating the time-fractional heat equation, i.e. where
the time derivative is a Caputo derivative of fractional order, Schneider introduced the notion of grey Brownian motion. The link between grey Brownian motion
and fractional differential equations were also studied later by \cite{Mur08} and \cite{Pag12}.

The generalized grey brownian motion and the Non-Gaussian framework in infinite dimensional analysis to which it belongs, i.e. Mittag-Leffler Analysis, were founded in
\cite{GrotJahnI} and \cite{GrotJahnII}.

In these papers, the authors define the spaces of test functions and distributions in Mittag-Leffler analysis and also they develop the principal tools proper to the field of infinite dimensional calculus, characterization theorems and integral transforms.\\
In this article, we prove the existence of test functions and distributions spaces for the Gamma grey measure and we establish the characterization theorems and tools for the analysis of the corresponding distribution spaces, Section \ref{sec2}. Moreover, we give explicit examples of the use of this characterization on Donsker's delta function, Noise distribution and Ornstein-Uhlenbeck process, respectively in Section \ref{sec3}, \ref{sec4} and \ref{sec5}.

\section{The Gamma grey noise and its calculus}
\label{sec2} 

By the results in \cite{Beg22}, we know that the Upper Incomplete Gamma function is the Laplace transform of $f_{\rho}(y)=\mathbbm{1}_{(1,\infty)}(y)G^{0,1}_{1,1}\big[y\big|^\rho _1\big]$, for $\rho \in (0,1]$, where $G^{m,n}_{p,q}$ is the Meijer G-function, see e.g.~the appendix in \cite{Beg22} for the special definition in the particular case. We use the translator operator and the properties of Laplace transform and its inverse transform (hereinafter denoted by  $\mathcal{L}(f(\cdot))(s)$ and $\mathcal{L}^{-1}(F(\cdot))(x)$, respectively), such that 
\begin{equation}\label{defdensity} f_{\rho,\theta}(x)=\mathcal{L}^{-1}\left( \frac{\Gamma(\rho,\theta+\cdot)}{\Gamma(\rho,\theta)}\right)(x)=\frac{1}{\Gamma(1-\rho)\Gamma(\rho,\theta)} \frac{e^{-\theta x}}{x(x-1)^{\rho}}, \quad x > 1, \end{equation}
for $\rho \in (0,1]$ and $\theta>0$.\\
In \cite{Beg22}, the authors studied the case of Gamma-Grey space $(\mathcal{S}', \mathcal{B},\nu_{\rho,\theta})$ and showed that the Gamma Grey measure satisfy the following properties:

\begin{enumerate}
	\item[P1] For $\rho \in (0,1]$ and $\theta > 0$, $\nu_{\rho,\theta}$ has an
	analytic Laplace transform in a neighborhood of zero $\mathcal{U}_{\theta}\subset \mathcal{S}_{%
		\mathbb{C}}$:
	\begin{equation}\label{Laptransf}
		\mathcal{S}_{\mathbb{C}} \supset \mathcal{U}_{\theta} \ni \phi \mapsto \ell_{\nu_{\rho,\theta}} (\phi):=\int_{\mathcal{S%
			}^{\prime}}\exp{\langle \omega,\phi \rangle }d\nu_{\rho,\theta}(\omega)=\frac{\Gamma(\rho,\theta -\frac{1}{2}\langle \phi,\phi \rangle)}{\Gamma(\rho,\theta)}.
	\end{equation}%
	\newline
	
	\item[P2] For $\rho \in (0,1]$ and $\theta > 0$, $\nu_{\rho,\theta}(\mathcal{U}%
	)>0$ for any non-empty open subset $\mathcal{U}\subset \mathcal{S}^{\prime}$.%
	\newline
\end{enumerate}
where $\mathcal{U}_{\theta}=B_{\sqrt{\theta}}(0)\oplus i B_{\sqrt{\theta}}(0)$, and $B_{\sqrt{\theta}}(0)=\{ \phi \in \mathcal{S} | \, \|\phi\|_{L^2}<\sqrt{\theta} \}$.

\begin{remark}
	The upper Incomplete Gamma function can be represented also by the Mellin-Barnes integral as follows:
	
	\[ \Gamma(\rho)\Gamma(\rho,z)=\frac{1}{2 \pi i}\int_{\mathcal{L}}\frac{\Gamma(\rho+s)\Gamma(s)}{\Gamma(s+1)}x^{-s}\mathrm{d}s=H_{2,0}^{1,2}\left[z \, \Bigg| \genfrac{}{}{0pt}{}{(1,1)}{(0,1),(\rho,1)}\right],  \]
	for $z \in \mathbb{C}\backslash\{0\}$, see Equation (2.10) in \cite{BG20}.\\
	By Theorem 1.3 in \cite{Kil04}, the above $H$-function has a series representation on $\mathbb{C} \backslash \{0\}$. For $z=0$, we have that $\Gamma(\rho,z)=\Gamma(\rho)$.\\
	So that the the Laplace transform $\ell_{\nu_{\rho,\theta}}(\phi)$ is well-defined for all $\phi \in \mathcal{S}_{\mathbb{C}}$.
\end{remark}

\begin{remark}
	Let $\rho \in (0,1]$, $\theta >0$ and $\lambda \in \mathbb{R%
	}\backslash \{0\}$, then the function $$\phi \mapsto \int_{\mathcal{S%
		}^{\prime}}\exp{\lambda \langle \omega,\phi \rangle }d\nu_{\rho,\theta}(\omega)$$  is analytic for $\phi \in \mathcal{S}_{\mathbb{C}}$. If $\lambda=0$, then Equation (\ref{Laptransf}) holds for $\phi \in \mathcal{S}_{\mathbb{C}}$.
\end{remark}

Using the identity theorem for analytic functions, we extend the function $$\lambda \mapsto \int_{\mathcal{S%
	}^{\prime}}\exp(\lambda \langle \omega,\phi \rangle ) d\nu_{\rho,\theta}(\omega)$$ from $\lambda \in \mathbb{R}$ to $z \in \mathbb{C}$, such that that we obtain the following lemma.
\begin{lemma}
	\label{Lap in 1} Let $\rho \in (0,1]$, $\theta >0$ and $z \in \mathbb{C
	}$, then the exponential function $$\mathcal{S}^{\prime }\ni \omega
	\mapsto e^{|z \langle x,\phi \rangle |}$$ is integrable and
	\begin{equation}
		\int_{\mathcal{S}^{\prime }}e^{z \langle x,\phi \rangle }d\nu _{\rho
			,\theta }(x)=\frac{\Gamma (\rho ,\theta -\frac{z ^{2}}{2}\langle \phi
			,\phi \rangle )}{\Gamma (\rho ,\theta )},\qquad \text{ for }\phi \in \mathcal{S}_{\mathbb{C}}.  \label{le}
	\end{equation}\\
\end{lemma}

In particular, for $z=i$, we have that 

\begin{equation} \label{CFonSC}
	\int_{\mathcal{S}^{\prime }}e^{i \langle x,\phi \rangle }d\nu _{\rho
		,\theta }(x)=\frac{\Gamma (\rho ,\theta +\frac{1}{2}\langle \phi
		,\phi \rangle )}{\Gamma (\rho ,\theta )},\qquad \text{ for }\phi \in \mathcal{S}_{\mathbb{C}}. 
\end{equation}






We have that Appell System exists for the Gamma Grey measure, see \cite{Beg22}. By Lemma 4.12, Sec. 5 and Sec. 6 in \cite{Kon98}, the test function space, i.e. $(\mathcal{S})^{1}_{\nu_{\rho,\theta}}$, and the distribution space, i.e. $(\mathcal{S})^{-1}_{\nu_{\rho,\theta}}$, exist, and we have:
\[  (\mathcal{S})^{1}_{\nu_{\rho,\theta}}\subset L^2(\nu_{\rho,\theta}) \subset(\mathcal{S})^{-1}_{\nu_{\rho,\theta}}\]
endowed with the dual pairing $\langle\langle \cdot,\cdot \rangle \rangle_{\nu_{\rho,\theta}}$ between $ (\mathcal{S})^{-1}_{\nu_{\rho,\theta}} $ and $(\mathcal{S})^{1}_{\nu_{\rho,\theta}}$ which is the bilinear extension of the inner product of $L^2(\nu_{\rho,\theta})$. \\
We define the $S_{\nu_{\rho,\theta}}$-transform by means of the normalized exponential $e_{\nu_{\rho,\theta}}(\cdot, \xi)$:

\[ (S_{\nu_{\rho,\theta}}\Phi )(\xi):= \langle \langle \Phi, e_{\nu}(\xi,\cdot) \rangle \rangle_{\nu_{\rho,\theta}}:=\frac{\Gamma(\rho,\theta)}{\Gamma(\rho,\theta - \frac{1}{2} \langle \xi,\xi\rangle)}\int_{S'}e^{\langle \omega, \xi\rangle}\Phi(\omega)\nu_{\rho,\theta}(d\omega), \quad \xi \in  U_{p,q}. \]

\begin{remark}
	The holomorphy of $\ell_{\nu_{\rho,\theta}}(\cdot)$ in $\mathcal{U}_{\theta}$ and $\ell_{\nu_{\rho,\theta}}(0)=1$, can be obtained by considering an $r>0$ such that $\ell_{\nu_{\rho,\theta}}(\xi)>0$ for each $\xi \in B_r(0)\subset \mathcal{U}_{\theta}$, and $q \in \mathbb{N}$ such that $r<2^{-q}$. By the embedding property of the nuclear spaces, we have that there exists $p \in \mathbb{N}$ such that $\|\xi \|_p \leq \|\xi\|_{L^2}$. Hence, we define $U_{p,q}=\{\xi \in \mathcal{S}_{\mathbb{C}}| \, 2^q \|\xi\|_p <1\}$, thus the normalized exponential is a test function of finite order and we can introduce the S-transform, see Example 6 and Section 7 in \cite{Kon98}.
\end{remark}

The properties (P1) and (P2) and the previous remark allow us to state the following theorem, which is a special case of Thm 8.34 in \cite{Kon98}.
\begin{theorem}\label{thmStransisomor}
	The $S_{\nu_{\rho,\theta}}$-transform is a topological isomorphism from $(\mathcal{S})^{-1}_{\nu_{\rho,\theta}}$ to $\text{Hol}_{0}(\mathcal{S}_{\mathbb{C}})$.
\end{theorem}

The above characterization theorem leads directly to two corollaries for integrals of elements in $(\mathcal{S})^{-1}_{\nu_{\rho,\theta}}$ in a weak sense and the convergence of sequences in $(\mathcal{S})^{-1}_{\nu_{\rho,\theta}}$.
\begin{theorem}[Thm.~4.10~in \cite{GrotJahnI}]\label{thmIntFunc}
	Let $(T, \mathcal{A},\mu)$ be a measurable space and $\Phi_t \in (\mathcal{N})^{-1}_{\nu_{\rho,\theta}}$ for all $t \in T$. Let $\mathcal{U}\subset \mathcal{N}_{\mathbb{C}}$ be an appropriate neighbourhood of zero and $C>0$, such that:
	\begin{itemize}
		\item[i)] $(S_{\nu_{\rho,\theta}}\Phi_{\cdot})(\xi):T \to \mathbb{C}$ is measurable for all $\xi \in \mathcal{U}$;
		\item[ii)] $\int_T |(S_{\nu_{\rho,\theta}}\Phi_t)(\xi)|\mathrm{d}\mu(t)\leq C$ for all $\xi \in \mathcal{U}$.
	\end{itemize}
	Then there exists $\Psi \in (\mathcal{S})^{-1}_{\nu_{\rho,\theta}}$ such that for all $\xi \in \mathcal{U}$
	\[ S_{\nu_{\rho,\theta}}\Psi(\xi) = \int_T S_{\nu_{\rho,\theta}}\Phi_t(\xi)\mathrm{d}\mu(t).  \]
	We denote $\Psi$ by $\int_T \Phi_t \mathrm{d}\mu(t)$ and call it the weak integral of $\Phi$.
\end{theorem}

\begin{proof}
	The proof is similar to Theorem 4.10 in \cite{GrotJahnI}.
\end{proof}

\begin{theorem}\label{thmStransConverg}
	Let $\{\Phi_n\}_{n\in \mathbb{N}}$ be a sequence in $(\mathcal{S})_{\nu_{\rho,\theta}}^{-1}$. Then $\{\Phi_n\}_{n \in \mathbb{N}}$ converges strongly in $(\mathcal{S})_{\nu_{\rho,\theta}}^{-1}$ if and only if there exist $p,q \in \mathbb{N}$ with the following two properties:
	\begin{itemize}
		\item[i)] $\{(S_{\nu_{\rho,\theta}}\Phi_n) (\xi)\}_{n \in \mathbb{N}}$ is a Cauchy sequence for all $\xi \in U_{p,q}$;
		\item[ii)] $S_{\nu_{\rho,\theta}}(\Phi_n)$ is holomorphic on $U_{p,q}$ and there is a constrant $C>0$ such that
		\[ |S_{\nu_{\rho,\theta}}(\Phi_n) (\xi)|\leq C  \]
		for all $\xi \in U_{p,q}$ and for all $n \in \mathbb{N}$.
	\end{itemize} 
\end{theorem}

\begin{proof}
	The proof is similar to Theorem 2.12 in \cite{GrotJahnII}.
\end{proof}
\begin{definition}
	For $\Phi \in (\mathcal{S})^{-1}_{\nu_{\rho,\theta}}$ and $\xi \in U_{p,q}=\{ \xi \in \mathcal{S}_{\mathbb{C}} |\; 2^q |\xi|^2_{p}<1 \}$, we define the $T_{\nu_{\rho,\theta}}$-transform by
	\[  (T_{\nu_{\rho,\theta}}\Phi)(\xi)=\langle \langle \Phi, e^{i\langle \cdot, \xi \rangle } \rangle \rangle_{\nu_{\rho,\theta}}=\int_{S'}e^{i\langle \omega, \xi\rangle}\Phi(\omega)\nu_{\rho,\theta}(d\omega). \]
\end{definition}
\begin{remark}
	We recall that if $\xi=0$, then $\exp{(i \langle \cdot, \xi \rangle)}=1$ and we have 
	\[ (T_{\nu_{\rho,\theta}}\Phi)(0)=\mathbbm{E}_{\nu_{\rho,\theta}}(\Phi). \]
\end{remark}
\section{Donsker's Delta}
\label{sec3} 
In the following section we make use of the Tricomi confluent hypergeometric function

\[  \Psi(a,c;z)= \frac{\Gamma(1-c)}{\Gamma(1+a-c)} \, _1F_1(a;c;z) + \frac{\Gamma(c-1)}{\Gamma(a)}z^{1-c} \, _1F_1(1+a-c;2-c;z), \]
where $_1F_1$ is the Kummer confluent hypergeometric function. The function $_1F_1$ is a special case of the generalized hypergeometric function $_pF_q$ with $p=1$ and $q=1$, defined as follows
\[ _pF_q\left(\genfrac{}{}{0pt}{}{a_1,\dots, a_p}{b_1,\dots, b_q};z\right)=\sum_{k\geq 0}\frac{(a_1)_k,\dots, (a_p)_k}{(b_1)_k,\dots, (b_q)_k}\frac{z^k}{k!}, \quad z \in \mathbb{C} \]
where $a_i, b_j \in \mathbb{C}\backslash \mathbb{Z}^{-}$ for $i=1,\dots,p$ and $j=1,\dots,q$ and the notation $(\cdot)_k$ is for Pochhammer symbol, see  \cite{PrudIV}.\\
We will need the following lemma to compute the Donsker's Delta, see \cite{Beg22}.\\
\begin{lemma}\label{lemdonsker}
	Let be $\rho \in (0,1)$, $\theta>0$.
	Let $\eta \in L^2$, $s\in \mathbb{R}$ and $\xi \in U_{p,q}$,
	then \[ T_{\nu_{\rho,\theta}}(e^{is\langle \cdot, \eta\rangle })(\xi)=\frac{1}{\Gamma(\rho,\theta)}\Gamma(\rho, \theta + S(s;\eta,\xi)), \]
	where $S(s;\eta,\xi)=\frac{1}{2}(s^2\langle \eta,\eta \rangle + \langle \xi,\xi\rangle +s\langle \eta,\xi \rangle)$.\\
	Furthermore, it is absolutely integrable over $\mathbb{R}$, for $\xi \in U_{p,q}$.
\end{lemma}

\begin{proof}
	Let $\xi \in U_{p,q}$ and $\eta \in L^2$, $s\in \mathbb{R}$, we have that $s\eta+\xi \in L^2_{\mathbb{C}}$, then
	\begin{eqnarray}\label{eqnpf}
		T_{\nu_{\rho,\theta}}(e^{is\langle \cdot, \eta\rangle})(\xi)&=\notag \int_{\mathcal{S}'} e^{i\langle\omega, s \eta + \xi \rangle}\nu_{\rho,\theta}(d\omega)\\
		&=\dfrac{\Gamma\big(\theta+\frac{1}{2}\langle s\eta + \xi, s\eta + \xi\rangle\big)}{\Gamma(\rho,\theta)},\\
	\end{eqnarray}
	where the equality in \ref{eqnpf} holds for Equation (\ref{CFonSC}).\\
	We define $$S(s;\eta,\xi):=\frac{1}{2}\langle s\eta + \xi, s\eta + \xi\rangle=\frac{1}{2}(s^2\langle\eta,\eta \rangle+\langle \xi,\xi \rangle + 2s\langle \eta,\xi\rangle)$$, we have that $S(s;\eta,\xi)\geq 0$ for each $s \in \mathbb{R}$ and $\xi \in \mathcal{S}_{\mathbb{C}}$, $\eta \in L^2$.\\ 
	Note that $\Gamma(\rho,z) \in \mathbb{C}$ for $z \in \mathbb{C}$, and $| \Gamma(\rho,z)|< C_{\rho} \Gamma(\rho,\Re(z))$ with $C_{\rho}>0$ and $\Re(z)>0$, see the appendix.\\
	Thus, the have that $$| \Gamma(\rho,\theta + S(s;\eta,\xi))|< C_{\rho} \Gamma(\rho,\theta+\Re(S(s;\eta,\xi)))$$, where $$\Re(S(s;\eta,\xi))=\langle\xi_1 + s \eta, \xi_1 + s \eta \rangle - \langle \xi_2,\xi_2\rangle,$$ with $\langle\xi_1 + s \eta, \xi_1 + s \eta \rangle=\|\xi_1 + s \eta \|^2$, which is positive for each $s$ and $\|\xi_2\|^2<\theta$.\\
	Hence, we can apply Bernstein theorem on $\Gamma(\rho,\theta + \Re(S(s;\eta,\xi)))$ for each $s \in \mathbb{R}$ on the line below denoted by $*$, see \cite{Schil12}.\\
	We have that:
	
	\begin{eqnarray*}
		\int_{\mathbb{R}} |T_{\nu_{\rho,\theta}}(e^{is\langle \cdot, \eta\rangle})(\xi)| \mathrm{d}s&=&\frac{1}{\Gamma(\rho,\theta)} \int_{\mathbb{R}} |\Gamma(\rho,\theta + \frac{1}{2}S(s;\eta,\xi))|\mathrm{d}s \\
		&\leq&\frac{C_\rho}{\Gamma(\rho,\theta)}\int_\mathbb{R}\Gamma(\rho,\theta +\Re( \frac{1}{2}S(s;\eta,\xi)))\mathrm{d}s\\
		&\overset{*}{=}&\frac{C_\rho}{\Gamma(\rho,\theta)} \int_{\mathbb{R}} \int_{0}^\infty e^{-r\Re(S(s;\eta,\xi))}f_{\rho,\theta}(r)\mathrm{d}r\mathrm{d}s \\
		&=&\frac{C_\rho}{\Gamma(\rho,\theta)}\int_0^\infty f_{\rho,\theta}(r) \int_{\mathbb{R}}e^{-r\Re(S(s;\eta,\xi))} \mathrm{d}s\mathrm{d}r\\
		&=&\frac{C_\rho \sqrt{2\pi}}{\Gamma(\rho,\theta)\sqrt{\langle \eta,\eta \rangle}}\int_0^\infty  f_{\rho,\theta}(r) r^{-1/2}e^{-\frac{r}{2} \left( \|\xi_1\|^2 - \|\xi_2\|^2\right)+ \frac{\langle \eta,\xi_1 \rangle}{\langle \eta,\eta\rangle} r}  \mathrm{d}r\\
		&=&\frac{\sqrt{2\pi}C_\rho}{\Gamma(1-\rho)\Gamma(\rho,\theta)\sqrt{\langle \eta,\eta \rangle}}\int_1^\infty \frac{e^{-\theta r}}{r(r-1)^{\rho}}r^{-1/2}e^{-\frac{r}{2} \left( \|\xi_1\|^2 -\|\xi_2\|^2-\frac{\langle \eta,\xi_1\rangle^2}{\langle \eta,\eta\rangle}     \right)}\mathrm{d}r  \\
		&<&\frac{\sqrt{2\pi}C_\rho}{\Gamma(\rho,\theta)\sqrt{\langle \eta,\eta \rangle}}\int_1^\infty e^{-\theta r/2} r^{-1} (r-1)^{-\rho} \mathrm{d}r<\infty
	\end{eqnarray*}
	
	where in the last inequality we use the fact that for $\xi \in U_{p,q}$
\end{proof}
Indeed we have
\[ \|\xi_1\|^2-\|\xi_2\|^2 -\frac{\langle \eta,\xi_1\rangle^2}{\langle \eta,\eta\rangle}>-\theta, \]
by Cauchy-Schwartz inequality.\\
The following representation holds for the Donsker's Delta in the space $(\mathcal{S})^{-1}_{\nu_{\rho,\theta}}$ as a weak integral in the sense of the Theorem \ref{thmIntFunc}:
\[ \delta(\langle \cdot, \eta \rangle)=\frac{1}{2\pi} \int_{\mathbb{R}} e^{is\langle \cdot, \eta \rangle}\mathrm{d}s. \]
\begin{remark}
	In view of Theorem \ref{thmIntFunc}, we choose $(\mathbb{R},\mathcal{B}(\mathbb{R}),\mathrm{d}s)$ as measurable space and $\Phi_s(\cdot)=e^{is\langle \cdot, \eta \rangle}$ for some $\eta \in L^2_{\mathbb{C}}$. The first assumption holds for the composition of measurable functions and the second one holds for Lemma \ref{lemdonsker}.\\ Moreover, we note that the S-transform and the T-transform differs for a constant on the complexification space.
\end{remark}
\begin{theorem}
	Let $\rho \in (0,1)$ and $\theta>0$. Let $\eta \in L^2_{\mathbb{C}}$ and $\xi \in U_{p,q}$, then
	\[ T_{\nu_{\rho,\theta}}(\delta(\langle \cdot, \eta \rangle))(\xi)=\frac{e^{-k}}{\sqrt{2\pi \langle \eta,\eta \rangle} \Gamma(\rho,\theta)} \Psi(1-\rho, 1/2 -\rho; k),\]
	where $k=\theta+\frac{1}{2}(\langle \xi,\xi\rangle -\frac{\langle \eta,\xi\rangle^2}{\langle \eta,\eta \rangle })$.
\end{theorem}

\begin{proof}
	By Lemma \ref{lemdonsker}, we have that for $\eta \in L^2_{\mathbb{C}}$, $s \in \mathbb{R}$ and $\xi \in  U_{p,q}$, we have 
	
	\begin{eqnarray*}
		T_{\nu_{\rho,\theta}}(\delta(\langle \cdot, \eta \rangle))(\xi)&=&\int_{\mathcal{S}'}e^{i \langle \omega,\xi \rangle} \delta(\langle \omega, \eta \rangle) \nu_{\rho,\theta}(\mathrm{d}\omega)\\
		&\overset{*}{=}&\frac{1}{2\pi} \int_{\mathbb{R}} \frac{\Gamma(\rho,\theta + S(s; \eta,\xi))}{\Gamma(\rho,\theta)}\mathrm{d}s\\
		&\overset{**}{=}&\frac{1}{2\pi} \int_{\mathbb{R}} \int_{0}^\infty e^{-rS(s;\eta,\xi)}f_{\rho,\theta}(r)\mathrm{d}r \mathrm{d}s\\
		&=&\frac{1}{2\pi} \int_{0}^\infty f_{\rho,\theta}(r) e^{-\frac{1}{2} r \langle \xi,\xi \rangle}\int_{\mathbb{R}}e^{-\frac{1}{2}rs^2 \langle \eta,\eta \rangle - rs \langle \eta,\xi\rangle}\mathrm{d}s\mathrm{d}r\\
		&=&\frac{1}{\sqrt{2\pi \langle \eta,\eta \rangle}}\int_0^\infty f_{\rho,\theta}(r) r^{-\frac{1}{2}} e^{-\frac{1}{2}r (\langle \xi,\xi\rangle -\frac{\langle \eta,\xi\rangle^2}{ \langle \eta,\eta \rangle } )}\mathrm{d}r\\
		&=&\frac{1}{\sqrt{2\pi \langle \eta,\eta \rangle} \Gamma(\rho,\theta)\Gamma(1-\rho)}\int_1^\infty e^{-rk}r^{-3/2} (r-1)^{-\rho}\mathrm{d}r  \\
		&=&\frac{e^{-k}}{\sqrt{2\pi \langle \eta,\eta \rangle} \Gamma(\rho,\theta)\Gamma(1-\rho)} \int_0^\infty e^{-ku}(u+1)^{-3/2}u^{-\rho} \mathrm{d}u\\
		&\overset{***}{=}& \frac{e^{-k}}{\sqrt{2\pi \langle \eta,\eta \rangle} \Gamma(\rho,\theta)} \Psi(1-\rho, 1/2 -\rho; k).
	\end{eqnarray*}
	where we use Theorem \ref{thmIntFunc} in the line denoted by $*$, use Equation 2.19 in \cite{MatSax2009} in $**$, use 1 of 2.1.3 in \cite{PrudIV} in $***$ and $\Psi(a,b;z)$ is the Tricomi confluent hypergeometric function, denote $S(s; \eta,\xi):=\frac{1}{2}(s^2\langle \eta,\eta \rangle + \langle \xi,\xi\rangle +2s\langle \eta,\xi \rangle)$, $k=\theta+\frac{1}{2}(\langle \xi,\xi\rangle -\frac{\langle \eta,\xi\rangle^2}{\langle \eta,\eta \rangle })>\theta$ by Cauchy-Schwartz inequality.\\
	
	
\end{proof}




\begin{corollary}
	For $\rho \in (0,1)$ and $\theta>0$, it holds
	\[  \mathbb{E}_{\nu_{\rho,\theta}}(\delta(\langle \cdot, \eta\rangle))=\frac{e^{-\theta}}{\sqrt{2\pi \langle \eta,\eta \rangle} \Gamma(\rho,\theta)} \Psi(1-\rho, 1/2 -\rho; \theta).\]
\end{corollary}

\begin{proof}
	
	For $\xi = 0$, we have 
	
	\begin{eqnarray*} 
		\mathbb{E}_{\nu_{\rho,\theta}}(\delta(\langle \cdot, \eta\rangle))&=&
		T_{\nu_{\rho,\theta}}(\delta(\langle \cdot, \eta \rangle))(0)\\
		&=&\frac{e^{-\theta}}{\sqrt{2\pi \langle \eta,\eta \rangle} \Gamma(\rho,\theta)} \Psi(1-\rho, 1/2 -\rho; \theta).\\
	\end{eqnarray*}
\end{proof}




\begin{theorem}
	Let $\eta \in L^2_{\mathbb{C}}$ and $a \in \mathbb{R}$. Then
	
	\[ \delta_a(\langle \cdot, \eta \rangle)=\lim_{n \to \infty}\frac{1}{2\pi}\int_{-n}^{n} e^{i s ( \langle \cdot, \eta\rangle -a )} ds \text{ in } \big( S\big)^{-1}_{\nu_{\rho,\theta}} \]
	and
	\[ \mathbb{E}(\delta_a(\langle \cdot, \eta \rangle))=\frac{e^{-\theta}}{\sqrt{2\pi \langle \eta, \eta \rangle}\Gamma(\rho,\theta)}\sum_{n\geq 0} \frac{(-1)^n}{n!}\frac{a^{2n}}{2^n \|\eta \|^{2n}} \Psi(1-\rho, \frac{1}{2} - n - \rho;\theta).
	\]

\end{theorem}

\begin{proof}
	Let $n\in \mathbb{N}$, $\Phi_n(\omega):= (2 \pi)^{-1} \int_{-n}^{n} e^{is(\langle \omega, \eta\rangle -a)}ds$  for almost all $\omega \in \mathcal{S}'$ and $\phi_n(x):=1_{[-n,n]}(x)e^{-ixa}\frac{\Gamma(\rho,\theta+S(s;\eta,\xi))}{\Gamma(\rho,\theta)}$ for $x \in \mathbb{R}$. We have $ \Phi_n \in L^2(\nu_{\rho,\theta})$ and $\phi_n \in L^1(\mathbb{R},dx)$ for each $n$ and $\phi_n$ converges to $\phi$ in $L^1(\mathbb{R},dx)$, as in Lemma \ref{lemdonsker}.\\

	We can apply dominated convergence to the $T_{\nu_{\rho,\theta}}$-transform of $\Phi_n$, so that for $\xi \in U_{p,q}$ we have 
	
	\begin{eqnarray*} 
		T_{\nu_{\rho,\theta}} (\delta_a(\langle \cdot, \eta \rangle))(\xi)&=&\frac{1}{2\pi\Gamma(\rho,\theta)} \int_{\mathbb{R}} e^{-isa} \Gamma(\rho, \theta + S(s,\eta,\xi))\mathrm{d}s \\ 
		&=&\frac{1}{2\pi\Gamma(\rho,\theta)} \int_{\mathbb{R}} e^{-isa} \Gamma(\rho, \theta + \frac{1}{2}(s^2\langle \eta,\eta \rangle + \langle \xi,\xi\rangle +2s\langle \eta,\xi \rangle))\mathrm{d}s\\
		&=&\frac{1}{2\pi} \int_{0}^\infty f_{\rho,\theta}(r) e^{-\frac{1}{2}r \langle \xi,\xi\rangle}\int_{\mathbb{R}} e^{-\frac{1}{2} r s^2 \langle \eta,\eta \rangle -s (r \langle \xi,\eta\rangle +ia) }\mathrm{d}s\mathrm{d}r\\
		&=& \frac{1}{\sqrt{2\pi \langle \eta, \eta \rangle}}\int_0^\infty f_{\rho,\theta} (r) r^{-1/2} e^{-\frac{1}{2} r \langle \xi,\xi\rangle } e^{\frac{r\langle \xi,\eta \rangle^2}{2\langle \eta,\eta \rangle}+\frac{ia\langle \xi,\eta \rangle }{\langle \eta,\eta \rangle}-\frac{a^2}{2\langle \eta,\eta \rangle r} }\mathrm{d}r\\
		&=&\frac{e^{\frac{ia\langle \xi,\eta \rangle }{\langle \eta,\eta \rangle}}}{\sqrt{2\pi \langle \eta, \eta \rangle}} \int_0 ^\infty e^{-r k} f_{\rho,\theta}(r)r^{-\frac{1}{2}} e^{-\frac{a^2}{2 r \langle \eta,\eta \rangle}}\mathrm{d}r \\
		&=&\frac{e^{\frac{ia\langle \xi,\eta \rangle }{\langle \eta,\eta \rangle}}}{\sqrt{2\pi \langle \eta, \eta \rangle }\Gamma(1-\rho)\Gamma(\rho,\theta)}\int_{1}^\infty e^{-r k-\frac{a^2}{2 \langle \eta,\eta \rangle}r^{-1}}r^{-3/2}\frac{e^{-\theta r}}{(r-1)^{\rho}}\mathrm{d}r\\
		&\overset{*}{=}&\frac{e^{\frac{ia\langle \xi,\eta \rangle }{\langle \eta,\eta \rangle}}}{\sqrt{2\pi \langle \eta, \eta \rangle}\Gamma(1-\rho)\Gamma(\rho,\theta)}\int_0^\infty e^{-(u+1) (k+\theta)-\frac{a^2}{2 \langle \eta,\eta \rangle}(u+1)^{-1}}(u+1)^{-3/2} u^{-\rho}\mathrm{d}u
	\end{eqnarray*}
	where $k=\frac{1}{2}\langle \xi,\xi \rangle - \frac{\langle \xi,\eta \rangle^2}{2\langle \eta,\eta \rangle}\geq0$ and with $u=r-1$.\\
	We use the analyticity of the function $\exp(k_1(u+1)^{-1})$ for $u>0$ in order to represent $T_{\nu_{\rho,\theta}} (\delta_a(\langle \cdot, \eta \rangle))(\xi)$ as a series.\\
	\begin{eqnarray*}
		T_{\nu_{\rho,\theta}} (\delta_a(\langle \cdot, \eta \rangle))(\xi)&=&k_2\sum_{n\geq 0} \frac{(-1)^n k_1^n}{n!} \int_0^\infty e^{-u(k+\theta)}(u+1)^{-n-3/2}u^{-\rho}\mathrm{d}u\\
		&\overset{**}{=}&k_2\sum_{n\geq 0} \frac{(-1)^n k_1^n}{n!} \Psi(1-\rho, \frac{1}{2} - n - \rho;k+\theta).
	\end{eqnarray*}
	
	where we use 1 of 2.1.3 in \cite{PrudIV} in $**$ and $\Psi(a,b;z)$ is the Tricomi confluent hypergeometric function, $k_1=\dfrac{a^2}{2 \|\eta \|^2}$ and  $k_2=\dfrac{e^{-k-\theta}e^{\frac{ia\langle \xi,\eta \rangle }{\langle \eta,\eta \rangle}}}{\sqrt{2\pi \langle \eta, \eta \rangle}\Gamma(\rho,\theta)}$.\\
	For $\xi=0$ we have that
	
	\begin{eqnarray*}
		\mathbb{E}(\delta_a(\langle \cdot, \eta \rangle))&=&T_{\nu_{\rho,\theta}} (\delta_a(\langle \cdot, \eta \rangle))(0)\\
		&=& \frac{e^{-\theta}}{\sqrt{2\pi \langle \eta, \eta \rangle}\Gamma(\rho,\theta)}\sum_{n\geq 0} \frac{(-1)^n}{n!}\frac{a^{2n}}{2^n \|\eta \|^{2n}} \Psi(1-\rho, \frac{1}{2} - n - \rho;\theta).\\
	\end{eqnarray*}
	For $\xi=0$,  we can represent the expectation of Donsker's Delta as:
	
	\begin{eqnarray*}
		T_{\nu_{\rho,\theta}}  (\delta_a(\langle \cdot, \eta \rangle))(0)&=&\frac{1 }{\sqrt{2\pi \langle \eta, \eta \rangle}} \int_0 ^\infty  f_{\rho,\theta}(r)r^{-\frac{1}{2}} e^{-\frac{a^2}{2 r \langle \eta,\eta \rangle}}\mathrm{d}r\\
		&=& \int_0 ^\infty \frac{ e^{-\frac{a^2}{2 r \langle \eta,\eta \rangle}}}{\sqrt{2\pi \langle \eta, \eta \rangle r}}f_{\rho,\theta}(r)\mathrm{d}r \\
		&=& \int_{-\infty}^{\infty} \delta_a(x)p(x;\eta,\rho,\theta) \mathrm{d}x \\
	\end{eqnarray*}
	where $p(x;\eta,\rho,\theta)$ is the density of a centered Gaussian r.v. with variance $R\|\eta \|^2$ and $R$ is a non negative random variable with density $f_{\rho,\theta}$.
\end{proof}

\begin{remark}
	We have that for $\rho=1$, the expectation of the Donsker's Delta on $(\mathcal{S}',\mathcal{B},\nu_{\rho,\theta})$ coincides with the one in the White Noise case.\\
	Indeed, for $\rho=1$, we have that $\Psi(0, \frac{1}{2} - n;\theta)=1$ (see Equation 13.6.3 in \cite{dlmf}) and $\Gamma(1,\theta)=e^{-\theta}$.
\end{remark}
\section{Grey Incomplete Gamma Noise}
\label{sec4} 

In order to present the Gamma Grey Brownian Motion, we introduce the fractional operator $M_{-}^{\alpha /2}$ defined, for
any $f\in \mathcal{S}$, as
\begin{equation*}
	M_{-}^{\alpha /2}f:=\left\{
	\begin{array}{l}
		\sqrt{C_{\alpha }}D_{\_}^{(1-\alpha )/2}f,\qquad \alpha \in (0,1) \\
		f,\qquad \qquad \qquad \qquad \alpha =1 \\
		\sqrt{C_\alpha}I_{\_}^{(\alpha -1)/2}f, \qquad \alpha \in (1,2)%
	\end{array}
	\right. ,
\end{equation*}
where
\begin{equation*}
	D_{\_}^{\beta }\,f(x):=-\frac{1}{\Gamma (1-\beta )}\frac{d}{\mathrm{d}x}
	\int_{x}^{\infty }f(t)(t-x)^{-\beta }\mathrm{d}t,\quad x\in \mathbb{R},\;\beta \in
	(0,1),
\end{equation*}
is the Riemann-Liouville fractional derivative and

\begin{equation*}
	I_{\_}^{\beta }\,f(x):=\frac{1}{\Gamma (\beta )} \int_{x}^{\infty
	}f(t)(t-x)^{\beta-1 }\mathrm{d}t,\quad x\in \mathbb{R},\;\beta \in (0,1),
\end{equation*}
is the Riemann-Liouville fractional integral, see \cite{GrotJahnI}.\\
In a similar way, we define the fractional operator $M_{+}^{\alpha/2}$, which is defined by means of the lower Riemann-Liouville fractional derivative and the lower Riemann-Liouville fractional integral.

The tempered $\Gamma$-grey Brownian motion  (hereafter $\Gamma$
-GBM) is defined as $B_{\alpha ,\rho }^{\theta }(t,\omega ):=\left\langle \omega ,M_{-}^{\alpha
	/2}\mathbbm{1}_{[0,t)}\right\rangle$ for $t\geq 0$ and a.e.-$\omega \in \mathcal{S}^{\prime
}( \mathbb{R})$ where $\alpha \in (0,2)$, $\rho \in (0,1]$ and $\theta>0$.\\
Thus, for $\xi \in \mathcal{U}_{\theta/2} \subset \mathcal{U}_{\theta}$, we can apply the $S$-transform on it and obtain

\begin{eqnarray}
	S_{\nu_{\rho,\theta}}(B_{\alpha ,\rho }^{\theta }(t,\cdot))(\xi)&=&\frac{\Gamma(\rho,\theta)}{\Gamma(\rho,\theta - \frac{1}{2}\|\xi\|^2)}\int_{\mathcal{S}'}\langle \omega ,M_{-}^{\alpha
		/2}\mathbbm{1}_{[0,t)}\rangle e^{\langle \omega, \xi\rangle } \nu_{\rho,\theta}(\mathrm{d}\omega) \label{momentlaplace}\\
	&=&\ell^{-1}_{\nu}(\xi) \int_{\mathcal{S}'} \frac{\partial}{\partial s} e^{\langle \omega, \xi\rangle + s\langle \omega,M_{-}^{\alpha
			/2}\mathbbm{1}_{[0,t)}  \rangle}\Big|_{s=0} \nu_{\rho,\theta}(\mathrm{d}\omega)\\
	&=&\ell^{-1}_{\nu}(\xi) \frac{\partial}{\partial s}\int_{\mathcal{S}'}  e^{\langle \omega, \xi + sM_{-}^{\alpha
			/2}\mathbbm{1}_{[0,t)}  \rangle} \nu_{\rho,\theta}(\mathrm{d}\omega)\Big|_{s=0},
\end{eqnarray}
where we can exchange the integral and the derivative due fact that the integral in line (\ref{momentlaplace}) is finite due to Cauchy-Schwartz inequality, $\Gamma$
-GBM is in $L^2(\nu_{\rho,\theta})$ and the Laplace transform (\ref{Laptransf}) is finite for $\xi \in \mathcal{U}_{\theta/2}$.

We have that $\int_{\mathcal{S}'}  e^{\langle \omega, \xi + sM_{-}^{\alpha
		/2}\mathbbm{1}_{[0,t)}  \rangle} \nu_{\rho,\theta}(\mathrm{d}\omega)= \frac{1}{\Gamma(\rho,\theta)} \Gamma(\rho, \theta + \frac{1}{2} \|\xi + sM_{-}^{\alpha
	/2}\mathbbm{1}_{[0,t)} \|^2)$, such that
\[ \frac{\partial}{\partial s}\int_{\mathcal{S}'}  e^{\langle \omega, \xi + sM_{-}^{\alpha/2}\mathbbm{1}_{[0,t)}\rangle} \nu_{\rho,\theta}(d\omega)\] \[\qquad \qquad = -e^{-(\theta + 1/2\| \xi + sM_{-}^{\alpha
		/2}\mathbbm{1}_{[0,t)} \|^2)} \Bigg( \theta + \frac{ \| \xi + sM_{-}^{\alpha
		/2}\mathbbm{1}_{[0,t)} \|^2}{2} \Bigg)^{\rho-1} \Big( s \|M_{-}^{\alpha
	/2}\mathbbm{1}_{[0,t)}\|^2 + \langle \xi, M_{-}^{\alpha
	/2}\mathbbm{1}_{[0,t)}  \rangle \Big), \]
thus
\[ \frac{\partial}{\partial s}\int_{\mathcal{S}'}  e^{\langle \omega, \xi + sM_{-}^{\alpha
		/2}\mathbbm{1}_{[0,t)}  \rangle} \nu_{\rho,\theta}(d\omega)\Big|_{s=0}= -e^{-(\theta + \| \xi \|^2)} \Bigg( \theta + \frac{ \| \xi  \|^2}{2} \Bigg)^{\rho-1} \Big( \langle \xi, M_{-}^{\alpha
	/2}\mathbbm{1}_{[0,t)}  \rangle \Big).\]

Hence, we get

\[ S_{\nu_{\rho,\theta}}(B_{\alpha ,\rho }^{\theta }(t,\cdot))(\xi)\]
\begin{equation}\label{StransfGammaGBM}
	\qquad \qquad=  \frac{\Gamma(\rho,\theta)}{\Gamma(\rho,\theta -\frac{1}{2}\|\xi\|^2)}(-e^{-(\theta + \| \xi \|^2)} )\Bigg( \theta + \frac{ \| \xi  \|^2}{2} \Bigg)^{\rho-1} \Big( \langle \xi, M_{-}^{\alpha
		/2}\mathbbm{1}_{[0,t)}  \rangle \Big). \end{equation}

Note that for $\alpha \in (0,2)$, $\langle \xi, M_{-}^{\alpha
	/2}\mathbbm{1}_{[0,t)}\rangle=\langle M_{+}^{\alpha/2}\xi,\mathbbm{1}_{[0,t)}\rangle=\int_{0}^tM_{+}^{\alpha/2}\xi(x)dx$, see (5.17) in \cite{skm93}.

Now we establish that the $\Gamma$-GBM is the differentiable in $(\mathcal{S})^{-1}_{\nu_{\rho,\theta}}$ and the existence of its noise in $(\mathcal{S})^{-1}_{\nu_{\rho,\theta}}$ is guaranteed applying Theorem \ref{thmStransConverg}.

\begin{theorem}
	Let $\alpha \in (0,2)$, $\rho \in (0,1]$ and $\theta>0$.
	We define the Grey Incomplete Gamma Noise in $(\mathcal{S})^{-1}_{\nu_{\rho,\theta}}$ as,
	
	\[ N_t^{\alpha,\rho,\theta}(\omega):=\lim_{h \to 0} \frac{B_{\alpha ,\rho }^{\theta }(t+h,\omega)- B_{\alpha ,\rho }^{\theta }(t,\omega) }{h}, \quad \text{a.e.-}\omega \in \mathcal{S}'\]
	and for every $\xi \in U_\theta$ we have
	
	\[  S_{\nu_{\rho,\theta}} \big(N_t^{\alpha,\rho,\theta}\big)(\xi)=-\frac{\Gamma(\rho,\theta)}{\Gamma(\rho,\theta -\frac{1}{2}\|\xi\|^2)}e^{-
		\theta - \| \xi \|^2}\Bigg( \theta + \frac{ \| \xi  \|^2}{2} \Bigg)^{\rho-1}(M_{+}^{\alpha
		/2} \xi)(t).\]

\end{theorem}

\begin{proof}
	For $t\geq 0$, we define the sequence 
	\[\mathscr{N}_{t,n}(\omega):=\frac{B_{\alpha ,\rho }^{\theta }(t+h_n,\omega)- B_{\alpha ,\rho }^{\theta }(t,\omega)}{ h_n} \in (\mathcal{S})^{-1}_{\nu_{\rho,\theta}}\]
	for a.e.-$\omega \in \mathcal{S}'$ and $n \in \mathbb{N}$  where $\{h_n\}_{n\in \mathbb{N}}$ are such that $h_n \to 0$ for $n \to \infty$.\\
	We apply the $S$-transform for $\xi \in U_\theta$
	
	\begin{eqnarray*}
		S_{\nu_{\rho,\theta}} \big(\mathscr{N}_{t,n} \big)(\xi)&=&\frac{1}{h_n}\Big( S_{\nu_{\rho,\theta}} (B_{\alpha ,\rho }^{\theta }(t+h_n))(\xi) - S_{\nu_{\rho,\theta}} (B_{\alpha ,\rho }^{\theta }(t))(\xi)\Big)  \\
		&=&-\frac{\Gamma(\rho,\theta)}{\Gamma(\rho,\theta -\frac{1}{2}\|\xi\|^2)}e^{-
			\theta - \| \xi \|^2}\Bigg( \theta + \frac{ \| \xi  \|^2}{2} \Bigg)^{\rho-1}\frac{ \langle \xi, M_{-}^{\alpha
				/2}1_{[t+h_n,t)}  \rangle }{h_n}.\\
	\end{eqnarray*}
	By the fact that
	\[h_n^{-1}|\langle \xi, M_{-}^{\alpha
		/2}1_{[t+h_n,t)}  \rangle|\leq h_n^{-1}\int_{t}^{t+h_n} |M_{+}^{\alpha
		/2} \xi(x)| dx\leq \max_{x \in \mathbb{R}}(|M_{+}^{\alpha
		/2} \xi(x)|)\]
	and that for $\xi \in U_\theta $
	
	\[ |-\frac{\Gamma(\rho,\theta)}{\Gamma(\rho,\theta -\frac{1}{2}\|\xi\|^2)}e^{-
		\theta -\| \xi \|^2}\Bigg( \theta + \frac{ \| \xi  \|^2}{2} \Bigg)^{\rho-1}   |     < e^{-\theta} \theta^{\rho-1}\] 
	
	we have that, for each $n \in \mathbb{N}$, and $\xi \in U_\theta $
	
	\[  |S_{\nu_{\rho,\theta}} \big(\mathscr{N}_{t,n} \big)(\xi)|<e^{-\theta} \theta^{\rho-1} \max_{x \in \mathbb{R}}(|M_{+}^{\alpha
		/2} \xi(x)|)<\infty ,\]
	we note that $M_{+}^{\alpha
		/2} \xi(\cdot) \in \mathcal{S}$.
	So we have,
	
	\begin{eqnarray*} 
		\lim_{n \to \infty}S_{\nu_{\rho,\theta}} \big(\mathscr{N}_{t,n} \big)(\xi)&=&\lim_{n \to \infty} -\frac{\Gamma(\rho,\theta)}{\Gamma(\rho,\theta -\frac{1}{2}\|\xi\|^2)}e^{-
			\theta - \| \xi \|^2}\Bigg( \theta + \frac{ \| \xi  \|^2}{2} \Bigg)^{\rho-1}\frac{ \langle \xi, M_{-}^{\alpha
				/2}1_{[t+h_n,t)}  \rangle }{h_n}\\
		&=&-\frac{\Gamma(\rho,\theta)}{\Gamma(\rho,\theta -\frac{1}{2}\|\xi\|^2)}e^{-
			\theta - \| \xi \|^2}\Bigg( \theta + \frac{ \| \xi  \|^2}{2} \Bigg)^{\rho-1}(M_{+}^{\alpha
			/2} \xi)(t).\\
	\end{eqnarray*}
	Thus, $\{S_{\nu_{\rho,\theta}} \big(\mathscr{N}_{t,n} \big)(\xi)\}_n$ is a Cauchy sequence for each $\xi$.\\
	By the fact that $S_{\nu_{\rho,\theta}} \big(\mathscr{N}_{t,n} \big)(\xi)$ is holomorphic on $U_\theta$ and it is finite for each $\xi \in U_\theta$, we can apply Theorem \ref{thmStransConverg} for the convergence of $\mathscr{N}_{t,n}$ to $\mathscr{N}_t$.
\end{proof}

\section{Grey Incomplete Gamma Ornstein-Uhlenbeck Process}
\label{sec5} 
We now compute the solution of Langevin Equation driven by Gamma Grey Brownian motion:

\begin{equation}\label{OUprocessdiffform}
	dX^{\alpha,\rho,\theta}(t)=-\lambda X^{\alpha,\rho,\theta}(t) dt + \kappa d B^{\theta}_{\alpha,\rho}(t), \quad  X_0^{\alpha,\rho,\theta}=x_0,  \end{equation}

where $\kappa \in \mathbb{R}$, $\lambda>0$ and $x_0 \in \ \mathbb{R}$.\\
We consider the weak integral form of Equation (\ref{OUprocessdiffform}):

\begin{equation}\label{OUprocessinteform}
	X^{\alpha,\rho,\theta}(t)=x_0-\lambda \int_0^tX^{\alpha,\rho,\theta}(s) ds + \kappa B^{\theta}_{\alpha,\rho}(t) \end{equation}
and we solve it applying the S-transform for $\xi \in U_\theta$ and using Equation (\ref{StransfGammaGBM}):

\[  S_{\nu_{\rho,\theta}} \Big( X^{\alpha,\rho,\theta}(t)  \Big)(\xi)= x_0-\lambda \int_0^tS_{\nu_{\rho,\theta}} \Big( X^{\alpha,\rho,\theta}(s)  \Big)(\xi)ds +\kappa C_{\rho,\theta}(\xi) \langle \xi, M_{-}^{\alpha
	/2}\mathbbm{1}_{[0,t)}  \rangle \]
where 
\[C_{\rho,\theta}(\xi):= -\frac{\Gamma(\rho,\theta)}{\Gamma(\rho,\theta -\frac{1}{2}\|\xi\|^2)}e^{-(\theta + \| \xi \|^2)}\Bigg( \theta + \frac{ \| \xi  \|^2}{2} \Bigg)^{\rho-1}.\]

By Theorem \ref{thmStransisomor}, we have that $x(s):=S_{\nu_{\rho,\theta}} \Big( X^{\alpha,\rho,\theta}(s)  \Big)(\xi)$ is an holomorphic function, so we can differentiate it:

\[  \frac{d}{dt} x(t)= -\lambda x(t) +k C_{\rho,\theta}(\xi) \frac{d}{dt} \langle \xi, M_{-}^{\alpha
	/2}\mathbbm{1}_{[0,t)}  \rangle  \]

and solve it:
\begin{eqnarray*}
	S_{\nu_{\rho,\theta}} \Big( X^{\alpha,\rho,\theta}(t)  \Big)(\xi)&=&x(t)\\
	&=& x_0 e^{-\lambda t} + \kappa C_{\rho,\theta }(\xi) \Big( \langle \xi, M_{-}^{\alpha
		/2}\mathbbm{1}_{[0,t)}  \rangle - \lambda \int_0^t e^{-\lambda(t-s)}\langle \xi, M_{-}^{\alpha
		/2}1_{[0,s)}  \rangle ds \Big)\\
	&\overset{*}{=}&  x_0 e^{-\lambda t} + \kappa C_{\rho,\theta }(\xi) \langle \xi,  h_{\alpha,t} \rangle,\\
\end{eqnarray*}

where we used the linearity of inner product and changing the order of integration in $*$ and
\[ h_{\alpha,t}(u)= \left(M_{-}^{\alpha
	/2}\mathbbm{1}_{[0,t)}\right)(u)  - \lambda \int_0^t e^{-\lambda(t-s)} \left(M_{-}^{\alpha
	/2}1_{[0,s)}\right)(u)ds.   \]

By considering that $ S_{\nu_{\rho,\theta}}\Big( \langle \cdot,  f \rangle \Big)(\xi)=\langle \xi,  f \rangle $ for $f \in \mathcal{S}_{\mathbb{C}}$ and that we can express $h_{\alpha,\lambda}$ as the limit of functions in Schwartz for density, we can invert the S-transform and we get:

\[X^{\alpha,\rho,\theta}(t,\omega)= x_0 e^{-\lambda t} + \kappa B^{\theta}_{\alpha,\rho}(t,\omega) - \lambda \kappa \langle \omega, \int_0^t e^{-\lambda(t-s)} M_{-}^{\alpha
	/2}1_{[0,s)}ds \rangle. \]

\begin{definition}
	The solution of the Langevin Equation (\ref{OUprocessdiffform}) is called "Grey Incomplete Gamma Ornstein-Uhlenbeck process" or $\Gamma$-gOU and its characteristic function for $t\geq 0$ is
	
	\[   \mathbb{E}\big( e^{iwX^{\alpha,\rho,\theta}(t)}\big)= \exp^{iwx_0e^{-\lambda t}} \frac{\Gamma(\rho,\theta+ \frac{w^2}{2} \| h_{\alpha,t} \|^2 )}{\Gamma(\rho,\theta)}. \]
\end{definition}

\section{Appendix}
\label{appendix}

\begin{lemma}
	Let $\rho \in (0,1)$ and $z \in \mathbb{C}$ such that $\Re(z)>0$, then 
	\[ |\Gamma(\rho,z)|<\frac{1}{\cos(\arg(z))^{\rho}} \Gamma(\rho,\Re(z)). \]
\end{lemma}
\begin{proof}
	By definition in \cite{dlmf}, we have that
	\[ \Gamma(\rho,z)=\int_{z}^{\infty} e^{-t}t^{\rho-1}dt \]
	such that the integral domain does not cross the real negative axis.\\
	We choose as integral domain the line $\mathcal{L}_{z}=\{w \in \mathbb{C}| w=re^{i \arg(z)}, \, r \in (|z|,+\infty)\}$.\\
	Applying $u=te^{-i\arg(z)}$ in what follows we have
	\begin{eqnarray*}
		|\Gamma(\rho,z)|&\leq&\int_{|z|}^\infty |e^{-u e^{i\arg(z)}}||(ue)^{i \arg(z)}|^{\rho-1}du\\
		&\leq&\int_{|z|}^\infty |e^{-u \cos(\arg(z))}||u|^{\rho-1}du\\
		&=&\frac{1}{\cos(\arg(z))^{\rho}}\int_{|z|\cos(\arg(z))}^\infty e^{-y}|y|^{\rho-1}dy,
	\end{eqnarray*}
	where in the last equality we use $y=u \cos(\arg(z))$.\\
	To conclude, we note that $\Re(z)=\Re(|z|e^{i\arg(z)})=|z|\cos(\arg(z))$.
\end{proof}

We note that for $\eta \in L^2$, $s \in \mathbb{R}$ and $\xi \in U_{p,q}$,  \[\Re(\theta+S(s;\eta,\xi))=\theta+\frac{1}{2}\| \xi_1 +s \eta\|^2 - \|\xi_2\|^2>0.\]
Hence, we can apply the above Lemma and note that

\[ \arg(\theta+S(s;\eta,\xi))=\arctan\left(\frac{2 \langle \xi_1+s\eta,\xi_2 \rangle}{\theta+\frac{1}{2}\| \xi_1 +s \eta\|^2 - \|\xi_2\|^2 }\right)\subset [-M,M], \]

for $s \in \mathbb{R}$ where $M\in(0,\pi/2)$. So we have that there exists $C_\rho>0$ such that 
\[ \frac{1}{\cos(\arg(z(s)))^\rho} < C_{\rho}, \quad s \in \mathbb{R}, \]

where $z(s)=\theta+S(s;\eta,\xi)$. We can also use $C_\rho:= \frac{1}{\cos(M)^{\rho}}$.

\section*{Acknowlegdements}
The authors thank J.~L.~da Silva for his valuable discussions about the project. L.C. wants to thank the Centro de Ci\^encias Matem\'aticas for the kind hospitality within his research stay in Funchal in spring 2023. We are grateful for the fruitful comments of L. Accardi which improved the study.

\end{document}